\documentclass[11pt]{article}

\usepackage[T1]{fontenc}
\usepackage[utf8]{inputenc}
\usepackage[english]{babel}
\usepackage{lmodern}
\usepackage{microtype}
\usepackage{geometry}
\usepackage{amsmath,amssymb,amsthm,mathtools}
\usepackage{enumitem}
\usepackage[colorlinks=true,linkcolor=blue,citecolor=blue,urlcolor=blue]{hyperref}
\hypersetup{pdftitle={Arithmetic intersections},pdfauthor={CombGeoLab}}

\newtheorem{theorem}{Theorem}[section]
\newtheorem{lemma}[theorem]{Lemma}
\newtheorem{proposition}[theorem]{Proposition}
\newtheorem{claim}{Claim}
\newtheorem{corollary}[theorem]{Corollary}
\theoremstyle{definition}

\newtheorem{conjecture}[theorem]{Conjecture}

\newcommand{\A}{\mathcal A}
\newcommand{\B}{\mathcal B}
\newcommand{\C}{\mathcal C}
\newcommand{\Ff}{\mathcal F}
\newcommand{\Pk}{\mathcal P}
\newcommand{\E}{\mathbb E}

\title{Families that intersect in an arithmetic progression}
\author{Andrey Kupavskii, Fedor Noskov, Georgy Sokolov, Nikolai Terekhov}
\date{}

\begin{document}
\maketitle

\begin{abstract}
For $n\in\mathbb N$, let $\mu_n$ be the uniform measure on $2^{[n]}$.
A family $\Ff\subseteq2^{[n]}$ is called $k$-AP-intersecting if
$F_1\cap F_2$ contains a nonconstant $k$-term arithmetic progression for every
$F_1, F_2\in\Ff$.  Writing $m_k$ for the supremum of
$\mu_n(\Ff)$ over all $n$ and all such families $\Ff$, we prove
\[
2^{-k}\le m_k\le32\cdot 2^{-k}\qquad(k\ge3).
\]
 This greatly improves a recent result of Gillott, who proved that $m_k \leq k^{-1/3}$ for all sufficiently large $k$, and answers his question on $m_k$'s exponential decay. We also prove a similar bound for the corresponding cross-intersection problem.
\end{abstract}

\section{Introduction}

For $n\in\mathbb N$, write $[n]=\{1,\dots,n\}$ and identify a subset of $[n]$
with its indicator vector in $\{0,1\}^n$.  Let $\mu_n$ be the uniform
probability measure on $2^{[n]}$.

Forbidden intersection problems are a classical topic in extremal combinatorics;
the Erd\H{o}s--Ko--Rado theorem~\cite{EKR} is a central example. In 1964, Katona~\cite{Katona1964}, answering one of the questions from~\cite{EKR}, proved a non-uniform analogue of the Erd\H{o}s--Ko--Rado theorem for $k$-intersecting families.

\begin{theorem} \label{t.Katona}
	Let $1\le k\le n$ and let $\mathcal{F}\subseteq2^{[n]}$ satisfy $|F_1 \cap F_2| \geq k$ for all $F_1, F_2 \in \mathcal{F}$. Then
	\begin{itemize}
		\item $|\mathcal{F}| \leq \sum_{i=s}^n {n \choose i}$ if $n + k = 2s$;
		\item $|\mathcal{F}| \leq {n - 1 \choose s} + \sum_{i=s + 1}^n {n \choose i}$ if $n + k = 2s + 1$.
	\end{itemize}
\end{theorem}

Katona's bound is attained by the ball
$\mathcal{F} = \{F\subseteq[n]: |F|\ge s\}$ when $n+k=2s$, and by a similar family in the other case. Note that if $k$ is fixed and $\mathcal{F}_n$ is an extremal $k$-intersecting family, we have
$$\lim_{n \rightarrow \infty} \mu_{n}(\mathcal{F}_n) = \frac{1}{2}.$$

In 1976, Simonovits and S\'os suggested several structured variants of the intersection problem for non-uniform families, in which trivial families are expected to be extremal. One of these variants studies families intersecting in a $k$-term arithmetic progression. A \emph{$k$-term arithmetic progression}, or $k$-AP, is a set $Q(a,d)=\{a,a+d,\dots,a+(k-1)d\}$,
where $a,d\in\mathbb N$, $d\ge1$, and $a+(k-1)d\le n$. Let $\Pk_k(n)$ be the collection of all such progressions. A family $\Ff\subseteq2^{[n]}$ is \emph{$k$-AP-intersecting} if
$F_1\cap F_2$ contains a member of $\Pk_k(n)$ for every $F_1, F_2\in\Ff$.
Set
\[
m_{k,n}:=\max\bigl\{\mu_n(\Ff):\Ff\text{ is $k$-AP-intersecting}\bigr\}.
\]
For fixed $k$, $m_{k,n}$ is nondecreasing in $n$: a family on $[n]$ extends to $[n+1]$ by allowing the new element freely. Since $m_{k,n}\le1$, the limit exists:
\[
m_k:=\lim_{n \rightarrow \infty}m_{k,n}.
\]

Simonovits and S\'os conjectured that the largest $3$-AP-intersecting family is trivial, that is, consists of all sets that contain a fixed $3$-term arithmetic progression. Their conjecture was recorded by Chung, Graham, Frankl and
Shearer~\cite[Conjecture~4]{CGFS1986}.

\begin{conjecture}[Simonovits--S\'os]\label{conj:three}
	For every $n\ge3$, a $3$-AP-intersecting family of subsets of $[n]$ has
	at most $2^{n-3}$ members. Equivalently, $m_3=1/8$.
\end{conjecture}

A natural generalization of this conjecture is $m_k = 2^{-k}$ for all $k \geq 3$.
Until recently, only the trivial upper bound $m_k\le1/2$ was known. Gillott~\cite[Theorems~3.3 and~2.1]{Gillott2026} introduced an adaptive concentration inequality extending Talagrand's work and obtained
\[
m_4\le\frac12-\varepsilon
\quad\text{for some absolute $\varepsilon>0$},
\qquad
m_k\le k^{-1/3}\quad\text{for sufficiently large $k$}.
\]
He asked whether $m_k$
decays exponentially in $k$~\cite[Question~4.1]{Gillott2026}.

More recently, Keevash~\cite{Keevash2026} proved that $m_3\le\frac12-c$
for an absolute constant $c>0$.

In this paper, we show that the trivial construction for $k$-AP-intersecting families is optimal up to an absolute multiplicative constant. Actually, we prove a bound for cross $k$-AP-intersecting families.  Families $\A,\B\subseteq2^{[n]}$ are \emph{cross $k$-AP-intersecting} if
$A\cap B$ contains a member of $\Pk_k(n)$ for every $A\in\A$ and $B\in\B$.
Set
\[
\widetilde m_{k,n}
:=\max\Bigl\{\mu_n(\A)\mu_n(\B):
\A,\B\subseteq2^{[n]}\text{ are cross $k$-AP-intersecting}\Bigr\},\ \ \ \ \ \
\widetilde m_k:=\sup_{n\ge1}\widetilde m_{k,n}.
\]

Our main result is the following theorem.
\begin{theorem}\label{thm:main}
	For every $k\ge4$ and every $n\ge1$,
	\[
	\widetilde m_{k,n}\le1024 \cdot 4^{-k}.
	\]
\end{theorem}

Since two copies of any $k$-AP-intersecting family $\Ff$ are cross $k$-AP-intersecting, we have $m_k \leq \widetilde m_k^{1/2}$. Thus, together with the lower bound provided by a trivial $k$-AP-intersecting family, Theorem \ref{thm:main} gives
$$2^{-k}\le m_k\le32 \cdot 2^{-k}\qquad(k\ge3),$$
where the upper bound for $k=3$ is immediate.

Another well-studied class of structured intersection patterns consists of the edge sets of graphs. Let $H$ be a graph with at least one edge. A family of graphs on the same vertex set is called $H$-intersecting if the common edges of any two members contain a copy of $H$. The most well-known case of the graph-intersection problem is the triangle-intersection problem that asks about the largest $K_3$-intersecting family. In 1986 Chung, Graham, Frankl and Shearer proved that every triangle-intersecting family has density at most $\frac{1}{4}$ using the entropy method \cite{CGFS1986}. Later Ellis, Filmus and Friedgut used Boolean analysis to prove that the trivial triangle-intersecting family is extremal \cite{EFF2012}. A similar result is known for $K_4$-intersecting families~\cite{BergerZhao2023}. For bipartite graphs $H$, until recently there was no non-trivial bound on $m_H = \sup\{\mu(\mathcal{F}): \text{$\mathcal{F}$ is $H$-intersecting} \}$. Gillott proved that $m_{C_4} < \frac{1}{2}$ using the same approach as in his bound on $m_4$ \cite{Gillott2026}. More recently, Han and Wang~\cite{HanWang2026} proved that $m_{P_4} < \frac{1}{2}$, where $P_4$ is a path with $4$ vertices. Since $m_H=1/2$ for every graph $H$ containing no $P_4$ or $K_3$, this characterizes the graphs with $m_H=1/2$.

Let us comment on the proof of Theorem~\ref{thm:main}. After taking up-closures and replacing
the second family by its family of complements, we obtain a decreasing family $\C$ such that
$A\setminus C$ contains a $k$-AP for every $A\in\A$ and $C\in\C$.  Now take a
uniform random set $X$ and, conditionally on $X$, choose a random subset
$Z=Z(X)\subseteq X$ and put
\[
   W(X):=\E\bigl(I(X\setminus Z\in\C)\mid X\bigr).
\]
The product $\mu_n(\A)\mu_n(\C)$ is bounded by
\begin{equation}\label{eq:rhs_master-abstract}
	\E\left[
	\frac{\mu_n(\C)}{W(X)}
	\E\bigl(I(Z\text{ contains a $k$-AP})\mid X\bigr)
	\right].
\end{equation}
See Lemma~\ref{lem:master} below.
Bounding \eqref{eq:rhs_master-abstract} is the core of the argument. We shall define $Z$ coordinate by coordinate, so that the first factor shall become the terminal value of a positive
supermartingale starting at $1$. 

The choice of $Z$ involves a trade-off. Adding more coordinates to $Z$ increases $W(X)$ and hence lowers
the term $\frac{\mu_n(\C)}{W(X)}$, but it also makes a progression inside $Z$ more likely. We carefully choose $Z$, which allows us to reduce $k$-AP bounds to $4$-AP bounds: we get a factor $1/4$ for each of the remaining $k-4$ terms.

To construct $Z$, we use a modification of an adaptive exposure procedure introduced by Gillott in the proof of \cite[Theorem~1.1]{Gillott2026}. Conditional on $X$, the elements of $Z$ are chosen independently, but their inclusion probabilities are determined adaptively. We reveal the coordinates of $X$ one by one and choose $\mathbb P[i\in Z\mid X]$ as a function of $X_1,\ldots,X_i$. This probability depends on the relative increase in the conditional density of $\C$ produced by deleting coordinate $i$.

In Section~\ref{sec:reduction}, we reduce the problem to bounding the right-hand side of \eqref{eq:rhs_master-abstract}. In Section~\ref{sec:process}, we construct $Z$. In Section~\ref{sec:supermartingale}, we analyze the supermartingale associated with this construction and derive the estimates used in the final step. In Section~\ref{sec:final}, we complete the proof of Theorem~\ref{thm:main}. Finally, in Section~\ref{sec:comparation}, we compare our proof with Gillott's~\cite{Gillott2026}.

\section{Reduction and the random-deletion inequality}\label{sec:reduction}

We first convert cross-intersection into a one-sided difference condition. Let $\A,\B\subseteq2^{[n]}$ be cross $k$-AP-intersecting. We may assume that $\mathcal{A}$ and $\mathcal{B}$ are up-sets, since up-closure preserves the property of being cross $k$-AP-intersecting. Set
$$\mathcal{C} = \{[n] \setminus B: B \in \mathcal{B}\}.$$
Then $\mu_n(\C)=\mu_n(\B)$ and $\C$ is a down-set. For $C=[n]\setminus B$,
we have $A\setminus C=A\cap B$, so $A\setminus C$ contains a $k$-AP for every
$A\in\A$ and $C\in\C$. Thus Theorem~\ref{thm:main} follows from the next statement.

\begin{theorem}\label{thm:reduced}
Let $k\ge4$, and let $\A,\C\subseteq2^{[n]}$ satisfy:
\begin{enumerate}[label=\textup{(\roman*)}]
\item $\A$ is an up-set and $\mathcal{C}$ is a down-set;
\item $A\setminus C$ contains a $k$-AP for every $A\in\A$ and $C\in\C$.
\end{enumerate}
Then
\[
   \mu_n(\A)\mu_n(\C)\le1024\cdot 4^{-k}.
\]
\end{theorem}

The following inequality explains the rest of the construction. We will guarantee that $W(X)>0$ in the next section.

\begin{lemma}[Random-deletion inequality]\label{lem:master}
Assume the hypotheses of Theorem~\ref{thm:reduced} and let $X$ be uniform on $2^{[n]}$.  Conditionally on
$X$, let $Z\subseteq X$ be any auxiliary random set, called the deletion set.  Define
\[
   W(X):=\E\bigl(I(X\setminus Z\in\C)\mid X\bigr).
\]
If $W(X)>0$ for every $X$, then
\begin{equation}\label{eq:master-abstract}
   \mu_n(\A)\mu_n(\C)
   \le
   \E\left[
      \frac{\mu_n(\C)}{W(X)}
      \E\bigl(I(Z\text{ contains a $k$-AP})\mid X\bigr)
   \right].
\end{equation}
\end{lemma}

\begin{proof}
Since $Z\subseteq X$, we have $Z=X\setminus(X\setminus Z)$. Therefore, if $X\in\A$ and $X\setminus Z\in\C$, then $Z$ contains a $k$-AP by Theorem~\ref{thm:reduced}(ii). Fixing $X$ and taking
conditional expectation over the auxiliary randomness therefore gives
\[
   I(X\in\A)W(X)
   \le
   \E\bigl(I(Z\text{ contains a $k$-AP})\mid X\bigr).
\]
Since $W(X)>0$, we may multiply by $\mu_n(\C)/W(X)$ and get
\[
   \mu_n(\C)\,I(X\in\A)
   \le
   \frac{\mu_n(\C)}{W(X)}
   \E\bigl(I(Z\text{ contains a $k$-AP})\mid X\bigr).
\]
Averaging over $X$ gives \eqref{eq:master-abstract}.
\end{proof}

\section{Adaptive construction of the deletion set}\label{sec:process}

Fix $\A,\C$ as in Theorem~\ref{thm:reduced}.
If $\mu_{n}(\mathcal{C})=0$, the theorem is trivial, so assume $\mu_{n}(\mathcal{C})>0$.

Let $X$ be uniform on $\{0,1\}^n$, and let
\[
   \Ff_i:=\sigma(X_1,\dots,X_i),\qquad 0\le i\le n,
\]
be the sigma-algebra generated by the first $i$ coordinates. We construct the deletion sets $Z^{(i)}$ and
the associated conditional probabilities $W_i$ recursively, starting with
\[
   Z^{(0)}=\varnothing,\qquad W_0=\mu_n(\C)>0.
\]
After step $n$, we set $Z := Z^{(n)}$.

Suppose that the first $i-1$ deletion rules have been defined and
$W_{i-1}>0$. Compare the two conditional probabilities
\begin{align}
   W^{(0)}_{i-1}
   &:=\E\bigl(I(X\setminus Z^{(i-1)}\in\C)
      \mid\Ff_{i-1},X_i=0\bigr),\label{eq:branch0}\\
   W^{(1)}_{i-1}
   &:=\E\bigl(I(X\setminus Z^{(i-1)}\in\C)
      \mid\Ff_{i-1},X_i=1\bigr).\notag
\end{align}
Since $X_i$ is a fair bit independent of $\Ff_{i-1}$, we have
\begin{equation*}
   W_{i-1}=\frac{W^{(0)}_{i-1}+W^{(1)}_{i-1}}2.
\end{equation*}
The earlier deletion rules do not depend on $X_i$, and $\C$ is a down-set, so
\begin{equation*}
   W^{(0)}_{i-1}\ge W^{(1)}_{i-1}.
\end{equation*}
Define the relative gap
\begin{equation*}
   r_i:=\frac{W^{(0)}_{i-1}-W^{(1)}_{i-1}}
              {W^{(0)}_{i-1}+W^{(1)}_{i-1}}\in[0,1].
\end{equation*}
Then $r_i$ is $\Ff_{i-1}$-measurable and
\begin{equation}\label{eq:branch-r}
   W^{(0)}_{i-1}=(1+r_i)W_{i-1},\qquad
   W^{(1)}_{i-1}=(1-r_i)W_{i-1}.
\end{equation}
The deletion rule is guided by how much deleting \(i\) increases the conditional probability that the remaining set belongs to \(\C\). If $X_i=0$, we cannot add $i$ to $Z$, since $Z\subseteq X$.
If $X_i=1$, keeping $i$ gives conditional probability
$W^{(1)}_{i-1}=(1-r_i)W_{i-1}$, whereas deleting it gives
$W^{(0)}_{i-1}=(1+r_i)W_{i-1}$, by \eqref{eq:branch-r}.
Thus $r_i$ measures the effect of deletion relative to $W_{i-1}$.
This relative effect is what matters for the reciprocal conditional
probability in \eqref{eq:master-abstract}. When $r_i$ is small, the two
branches are close, while deletion still makes a $k$-AP in $Z$ more likely.
We therefore let the deletion probability increase with $r_i$:
\begin{align*}
   Z^{(i)} := \begin{cases}
      Z^{(i - 1)} \cup \{i\} & \text{ with probability } p(r_i), \\
      Z^{(i - 1)} & \text{ with probability } 1- p (r_i),
   \end{cases} \qquad \text{ if } X_i = 1,
\end{align*}
where 
\begin{equation}\label{eq:p-function}
   p(r):=\frac{2r}{1+3r},\qquad 0\le r\le1.
\end{equation}
For brevity, we denote the probability of adding $\{i\}$ to $Z^{(i - 1)}$ by $p_i$:
\begin{align*}
   p_i = \begin{cases}
      0, & \text{ if } X_i = 0, \\
      p(r_i), & \text{ otherwise.}
   \end{cases}
\end{align*}

Let $W_i$ denote the conditional probability that $X\setminus Z^{(i)}$ belongs to
$\C$:
\begin{equation*}
   W_i:=\E\bigl(I(X\setminus Z^{(i)}\in\C)\mid\Ff_i\bigr).
\end{equation*}
This expectation averages over the auxiliary choices of $Z$ and the unexposed
coordinates $X_{i+1},\dots,X_n$. The next lemma verifies $W_i>0$,
completing the inductive step.
\begin{lemma}\label{lem:W-recursion}
For every $1\le i\le n$,
\begin{equation}\label{eq:W-recursion}
W_i=
\begin{cases}
   (1+r_i)W_{i-1},&X_i=0,\\[1mm]
   D(r_i)W_{i-1},&X_i=1
\end{cases},
\qquad
D(r):=1-r+2rp(r)=\frac{(1+r)^2}{1+3r}.
\end{equation}
In particular, $W_i>0$ for every $0\le i\le n$.
\end{lemma}

\begin{proof}
If $X_i=0$, coordinate $i$ is absent and cannot be deleted.  Thus, by
\eqref{eq:branch0},
\[
   W_i=W^{(0)}_{i-1}=(1+r_i)W_{i-1}.
\]
If $X_i=1$, we have
\begin{align*}
   W_i
   &=p(r_i)W^{(0)}_{i-1}+(1-p(r_i))W^{(1)}_{i-1}\\
   &=\bigl(1-r_i+2r_ip(r_i)\bigr)W_{i-1}.
\end{align*}
Substituting \eqref{eq:p-function} gives
$D(r)=(1+r)^2/(1+3r)$.  Since $W_0=\mu_{n}(\mathcal{C})>0$ and both factors in
\eqref{eq:W-recursion} are positive on $[0,1]$, positivity follows by induction.
\end{proof}

To analyze the recursion, we introduce the $\Ff_i$-measurable inverse weight
\begin{equation*}
   L_i:=\frac{\mu_{n}(\mathcal{C})}{W_i},\qquad 0\le i\le n.
\end{equation*}
In particular, as $W_0=\mu_n(\C)$ and $W_n=\E\bigl(I(X\setminus Z\in\C)\mid X\bigr)=W(X)$, it satisfies
\begin{equation}\label{eq:L-endpoints}
   L_0=1,
   \qquad
   L_n=\frac{\mu_{n}(\mathcal{C})}{W_n}=\frac{\mu_{n}(\mathcal{C})}{W(X)},
\end{equation}
and \eqref{eq:W-recursion} is equivalent to
\begin{equation}\label{eq:L-recursion}
L_i=
\begin{cases}
   \frac{1}{1+r_i}L_{i-1},&X_i=0,\\[1mm]
   \frac{1}{D(r_i)}L_{i-1},&X_i=1
\end{cases}.
\end{equation}
In the next section, we show that $(L_i)_{i=0}^n$ is a supermartingale
and establish the one-step estimates needed for the proof.

\section{The inverse weight supermartingale}\label{sec:supermartingale}

Recall that our goal is to prove
\begin{align*}
	\E\left[
	L_n(X)
	\E\bigl(I(Z\text{ contains a $k$-AP})\mid X\bigr)
	\right] \le1024\cdot 4^{-k}.
\end{align*}

We now establish the estimates on $(L_i)$ needed for this bound.
The choice $p(r)=2r/(1+3r)$ is used through three elementary identities.

\begin{lemma}[Local identities]\label{lem:local}
For $0\le r\le1$ define
\[
   \delta(r):=1-\frac{1}{2(1+r)}-\frac{1}{2D(r)}.
\]
Then
\begin{align}
   \delta(r)&=\frac{r^2}{(1+r)^2} \geq 0,\notag\\
   \frac{p(r)}{2D(r)}&=\frac{r}{(1+r)^2}\le\frac14,\label{eq:marked-cost}\\
   \frac{p(r)^2}{2D(r)}
      &=\frac{2r^2}{(1+r)^2(1+3r)}
        \le2\delta(r).\label{eq:energy-cost}
\end{align}
\end{lemma}

\begin{proof}
All identities follow by direct substitution.  For
\eqref{eq:marked-cost}, the inequality $r/(1+r)^2\le1/4$ is equivalent to
$(1-r)^2\ge0$.
\end{proof}

\begin{proposition}[Supermartingale and one-step costs]\label{prop:supermartingale}
The process $(L_i)_{i=0}^n$ is a positive supermartingale with respect to
the filtration $(\Ff_i)_{i=0}^n$.  More precisely,
\begin{equation}\label{eq:drift}
   L_{i-1}-\E(L_i\mid\Ff_{i-1})
      =\delta(r_i)L_{i-1}.
\end{equation}
Moreover,
\begin{align}
   \E(L_ip_i\mid\Ff_{i-1})
      &\le\frac14L_{i-1},\label{eq:one-mark}\\
   \E(L_ip_i^2\mid\Ff_{i-1})
      &\le2\bigl(L_{i-1}-\E(L_i\mid\Ff_{i-1})\bigr).
      \label{eq:one-energy}
\end{align}
\end{proposition}

\begin{proof}
Conditionally on $\Ff_{i-1}$, the bit $X_i$ is fair and $r_i$ is fixed.  Hence
\eqref{eq:L-recursion} gives
\begin{align*}
   \E(L_i\mid\Ff_{i-1}) &= \frac{1}{2}\left(\E(L_i\mid\Ff_{i-1}, X_i = 0) + \E(L_i\mid\Ff_{i-1}, X_i=1) \right) \\
   &=\left(\frac{1}{2(1+r_i)}+\frac{1}{2D(r_i)}\right)L_{i-1}
   =\bigl(1-\delta(r_i)\bigr)L_{i-1}.
\end{align*}
This proves \eqref{eq:drift} and the supermartingale property.

Since $p_i=0$ on $\{X_i=0\}$ and $p_i=p(r_i)$ on $\{X_i=1\}$,
\[
   \E(L_ip_i\mid\Ff_{i-1})
   =\frac{p(r_i)}{2D(r_i)}L_{i-1}
   \le\frac14L_{i-1}
\]
by \eqref{eq:marked-cost}.  Similarly, by \eqref{eq:energy-cost}
\[
   \E(L_ip_i^2\mid\Ff_{i-1})
   =\frac{p(r_i)^2}{2D(r_i)}L_{i-1}
   \le2\delta(r_i)L_{i-1},
\]
and \eqref{eq:one-energy} follows from \eqref{eq:drift}.
\end{proof}

Summing the quadratic estimate gives the following bound, used for the weighted $4$-AP sum.

\begin{corollary}\label{cor:tail-energy}
For every $0\le t<n$,
\begin{equation*}
   \E\left(\sum_{j=t+1}^nL_jp_j^2\,\middle|\,\Ff_t\right)
   \le2\bigl(L_t-\E(L_n\mid\Ff_t)\bigr)
   \le2L_t.
\end{equation*}
In particular,
\begin{equation}\label{eq:total-energy}
   \sum_{j=1}^n\E(L_jp_j^2)\le2.
\end{equation}
\end{corollary}

\begin{proof}
Taking conditional expectations in \eqref{eq:one-energy} given $\Ff_t$ and summing gives
\begin{align*}
\E\left(\sum_{j=t+1}^nL_jp_j^2\,\middle|\,\Ff_t\right)
&\le2\sum_{j=t+1}^n
\E\left(L_{j-1}-\E(L_j\mid\Ff_{j-1})\,\middle|\,\Ff_t\right)\\
&=2\sum_{j=t+1}^n
\left(\E(L_{j-1}\mid\Ff_t)-\E(L_j\mid\Ff_t)\right)\\
&=2\bigl(L_t-\E(L_n\mid\Ff_t)\bigr).
\end{align*}
Positivity of $L_n$ gives the second inequality.  Taking $t=0$ and using
$L_0=1$ proves \eqref{eq:total-energy}.
\end{proof}

\section{Completion of the proof}
\label{sec:final}

By Lemma~\ref{lem:W-recursion} and the definition of $W(X)$, $W_n=W(X)>0$;
by \eqref{eq:L-endpoints}, $L_n=\mu_n(\C)/W(X)$.
Lemma~\ref{lem:master} therefore gives
\begin{equation}\label{eq:master}
   \mu_n(\A)\mu_{n}(\mathcal{C})
   \le
   \E\left[
      L_n\E\bigl(I(Z\text{ contains a $k$-AP})\mid\Ff_n\bigr)
   \right].
\end{equation}
Since probabilities $p_i$ are fixed given $X$, for $S\subseteq[n]$, we define
\[
   P_S:=\E\bigl(I(S\subseteq Z)\mid\Ff_n\bigr)
   =\prod_{i\in S}p_i,
\]
where $P_\varnothing=1$.  Note that $P_S$ is $\Ff_{\max S}$-measurable when $S\ne\varnothing$.
Then, by the union bound,
\begin{equation}\label{eq:union}
   \E\left[
      L_n\E\bigl(I(Z\text{ contains a $k$-AP})\mid\Ff_n\bigr)
   \right]
   \le
   \sum_{Q\in\Pk_k(n)}\E(L_nP_Q).
\end{equation}
We bound the last sum by the following claim.

\begin{claim}
Consider a progression $Q(a, d) = \{a, a + d, \ldots, a + (k - 1)d \} \in \Pk_k(n)$ of length $k \geq 4$.  In the above setup, we have
\begin{align}
\label{eq: reduction to length-4 progressions}
   \E (L_n P_Q) \le 4^{- (k - 4)} \E (L_{a + 3d} \prod_{t = 0}^3 p_{a + td}).
\end{align}
\end{claim}

\begin{proof}
Since $P_{Q}$ is $\Ff_{a + (k - 1)d}$-measurable, the supermartingality of $(L_{i})_{i = 0}^{n}$ implies
\begin{align*}
   \E(L_n P_Q) = \E (P_Q \E(L_n \, | \, \Ff_{a + (k - 1)d})) \le \E (P_Q L_{a + (k - 1)d} ) = \E (L_{a + (k - 1)d} \prod_{t = 0}^{k - 1} p_{a + td}).
\end{align*}
Now, let $s$ be the least integer in $\{4, \ldots, k\}$ for which we have
\begin{align}
\label{eq: s-hypothesis}
   \E \left (L_{a + (k - 1)d} \prod_{t = 0}^{k - 1} p_{a + td} \right ) \le 4^{- (k - s)} \E \left (L_{a + (s - 1) d} \prod_{t = 0}^{s - 1} p_{a + t d} \right ).
\end{align}
If $s > 4$, then, conditioning on $\Ff_{a + (s - 1)d - 1}$ in the expectation on the right-hand side of \eqref{eq: s-hypothesis} implies
\begin{align*}
   \E \left (L_{a + (s - 1) d} \prod_{t = 0}^{s - 1} p_{a + t d} \right ) &  =  \E \; \left [ \E \left (L_{a + (s - 1) d} \prod_{t = 0}^{s - 1} p_{a + t d} \; | \; \Ff_{a + (s - 1)d - 1} \right ) \right ]  \\
   & = \E \; \left [ \prod_{t = 0}^{s - 2} p_{a + t d} \; \cdot \; \E \left (L_{a + (s - 1) d} \cdot  p_{a + (s - 1) d}  \; | \; \Ff_{a + (s - 1)d - 1} \right ) \right ] \\
   & \overset{\eqref{eq:one-mark}}{\le}  \frac{1}{4} \E \; \left [ \prod_{t = 0}^{s - 2} p_{a + t d} \; \cdot \; L_{a + (s - 1) d - 1} \right ]  \\
   & \le \frac{1}{4} \E \; \left [ L_{a + (s - 2) d}  \prod_{t = 0}^{s - 2} p_{a + t d}  \right ],
\end{align*}
where we used the supermartingality of $(L_i)_{i = 0}^n$ on the last line again. Substituting this bound into~\eqref{eq: s-hypothesis} implies that~\eqref{eq: s-hypothesis} holds for $s - 1$ as well, a contradiction. Therefore, we have $s = 4$, which is precisely~\eqref{eq: reduction to length-4 progressions}.
\end{proof}

As any progression of length $k \geq 4$ is uniquely determined by its first four terms, the claim together with~\eqref{eq:union}  implies
\begin{equation}\label{eq:k-to-four}
   \E\left[
      L_n\E\bigl(I(Z\text{ contains a $k$-AP})\mid\Ff_n\bigr)
   \right] 
   \le4^{-(k-4)}
   \sum_{R\in\Pk_4(n)}\E(L_{\max R}P_R).
\end{equation}
The last sum is bounded by the following lemma.

\begin{lemma}\label{lem:four}
One has
\begin{equation}\label{eq:four-bound}
   \sum_{R\in\Pk_4(n)}\E(L_{\max R}P_R)\le4.
\end{equation}
\end{lemma}

\begin{proof}
Write $R=\{a_1<a_2<a_3<a_4\}$.  We have
\[
   p_{a_1}p_{a_2}p_{a_3}p_{a_4}
   \le\frac12\left(p_{a_1}^2p_{a_2}^2+p_{a_3}^2p_{a_4}^2\right).
\]
For the first term, $p_{a_1}^2p_{a_2}^2$ is $\Ff_{a_2}$-measurable, so the
supermartingale property gives
\[
   \E(L_{a_4}p_{a_1}^2p_{a_2}^2)
   =\E\left[p_{a_1}^2p_{a_2}^2\E(L_{a_4}\mid\Ff_{a_2})\right]
   \le\E(L_{a_2}p_{a_1}^2p_{a_2}^2).
\]
Hence
\begin{align*}
\sum_{R\in\Pk_4(n)}\E(L_{\max R}P_R)
&\le\frac12\sum_{R\in\Pk_4(n)}
   \E(L_{a_2}p_{a_1}^2p_{a_2}^2)
 +\frac12\sum_{R\in\Pk_4(n)}
   \E(L_{a_4}p_{a_3}^2p_{a_4}^2)\\
&\le\sum_{1\le i<j\le n}\E(L_jp_i^2p_j^2).
\end{align*}
The last inequality uses the elementary fact that a $4$-AP is uniquely determined
by its first two terms and also by its last two terms; thus any ordered pair
$i<j$ occurs at most once in each of the two sums.

Now apply Corollary \ref{cor:tail-energy} twice.  Since $p_i^2$ is $\Ff_i$-measurable,
\begin{align*}
\sum_{1\le i<j\le n}\E(L_jp_i^2p_j^2)
&=\sum_{i=1}^{n-1}
  \E\left[p_i^2\,
  \E\left(\sum_{j=i+1}^nL_jp_j^2\,\middle|\,\Ff_i\right)\right]\\
&\le2\sum_{i=1}^{n-1}\E(L_ip_i^2) \le4,
\end{align*}
where the last inequality is \eqref{eq:total-energy}.  This proves
\eqref{eq:four-bound}.
\end{proof}

Now, we are ready to complete the proofs of Theorems~\ref{thm:reduced} and~\ref{thm:main}.

\begin{proof}[Proof of Theorem~\ref{thm:reduced} and Theorem~\ref{thm:main}]
Combining \eqref{eq:master}, \eqref{eq:union}, \eqref{eq:k-to-four}, and
\eqref{eq:four-bound}, we have
\begin{align*}
   \mu_n(\A)\mu_n(\C)
   \le4^{-(k-4)}\cdot4
   =1024\cdot 4^{-k}. & \qedhere
\end{align*}
\end{proof}

\section{Comparison with Gillott's proof} \label{sec:comparation}

Our approach builds on Gillott's framework~\cite{Gillott2026}. We highlight the differences that give exponential decay with the correct base in place of the polynomial bound.

First, our deletion probabilities depend on the relative increase in the conditional density of $\C$, whereas Gillott uses the absolute increase. This modification already suffices for an exponential bound.

Second, in order to bound
$$\E\left[ L_n(X) \E\bigl(I(Z\text{ contains a $k$-AP})\mid X\bigr) \right] = \E\left[ \frac{\mu_{n}(\C)}{W_n(X)} \E\bigl(I(Z\text{ contains a $k$-AP})\mid X\bigr) \right],$$
Gillott estimates the surviving density and the probability that $Z$ contains a
$k$-AP separately. He shows that the latter probability is small for independent
sampling with a sufficiently small $\ell_2$-norm of the probability vector
\cite[Theorem~2.1]{Gillott2026}, and that for most $X$ one can choose such a
deletion set with $W_n(X)$ close to $\mu_n(\C)$ from below
\cite[Theorem~1.1]{Gillott2026}. We instead estimate the weighted probability
$L_n(X)\E\bigl(I(Z\text{ contains a $k$-AP})\mid X\bigr)$ directly, using
\eqref{eq:one-mark} and \eqref{eq:one-energy} to retain the dependence between
$L_i$ and $p_i$.

We also remark that the pairing argument in Lemma~\ref{lem:four} can be viewed as an
elementary discrete Brascamp--Lieb-type inequality.

\section{Declaration of generative AI use}\label{sec:AI}

We first asked ChatGPT 5.6 to improve the $k$-AP result of \cite{Gillott2026}. It initially obtained the exponential bound $m_k \le \exp(-\Omega(k))$. After further discussion, during which we provided additional ideas, ChatGPT improved this bound to $m_k \le 2^{-k+o(k)}$. Subsequent conversations led to the stronger estimate $\widetilde{m}_k \le O(4^{-k})$. ChatGPT generated the initial draft of the manuscript. The authors verified the arguments, revised and edited the manuscript, and take full responsibility for the final content.

\end{document}